\documentclass[CEJM,PDF]{cej} 
\usepackage{layout}
\usepackage{color, amsmath, amssymb,float}
\theoremstyle{plain}

\newtheorem{con}{Conjecture}

\numberwithin{equation}{section}
\newcommand{\ba}{\textnormal{Ball}}
\newcommand{\orb}{\textnormal{Orb}}
\newcommand{\RE}{\textnormal{Re}}

\definecolor{darkblue}{rgb}{0.067,0.008,0.571}
\definecolor{darkgreen}{rgb}{0.008,0.871,0.067}
\definecolor{darkred}{rgb}{0.871, 0.008,0.067}
\usepackage[colorlinks,pdfusetitle,urlcolor=darkblue,citecolor=darkred,linkcolor=darkblue,bookmarksnumbered,plainpages=false]{hyperref}

\title{A Note on Weak Hypercyclicity and Linear Fractional Composition Operator}

\author{Arman Shokrollahi
       }

\shortauthor{Arman Shokrollahi}

\institute{Department of Mathematics, West Virginia University, Morgantown, WV 26506, USA.
          }

\abstract{This paper discusses the existence of a sufficient condition for an operator to be weakly hypercyclic. We establish a weak hypercyclicity criterion, and thereupon we can answer questions 5.3 and 5.8 posed by Chan and Sanders in \cite{5}. Lastly, we show that for specific type of composition operator, weak hypercyclicity and hypercyclicity are equivalent.
}

\keywords{Banach space \*\ chaos \*\  composition operator
 \*\ Hardy space \*\ hypercyclicity \*\ linear operator \*\ norm hypercyclic \*\ weak hypercyclicity.}

\msc{47A16; 47S99; 46A45}

\begin{document}
\maketitle

\section{Introduction}
In this note, we try to briefly discuss a sufficient condition in terms of norm and weakly open sets for an operator on a reflexive Banach space to be weakly hypercyclic. This gives a useful criterion for weak hypercyclicity of operators. We apply this to specific classes of weakly hypercyclic operators including bilateral weighted shifts on $\ell^p(\mathbb{Z})$ with $1\leq p<\infty$. This, in turn,  provides a large class of weakly hypercyclic bilateral shift which are not norm hypercyclic and it answers question 5.3 in \cite{5} in a different way than what Sanders presented in \cite{san}. In the last section, we show hypercyclicity and weak hypercyclicity are equivalent for a composition operator on the space $H(U)$ of all complex-valued functions holomorphic on the open unit disk $U$ or the Hardy space $H^2$ that is the collection of functions $f \in H(U)$ with $\sum_{n=1}^{\infty}| \hat{f}(n) |^2 < \infty$. So, there are other classes of operators for which norm hypercyclicity and weak hypercyclicity are equivalent and it answers question 5.8 in \cite{5}. The authors in \cite{2,3,4, 5, 7, 8, 9, 10, 11, 13} have already studied weak hypercyclicity and composition operators in various ways. Dodson \cite{dod} has also prepared a survey in this regard.

The following theorem summarizes all the necessary conditions which have been obtained for weakly hypercyclic operators. For a proof,  see \cite{5} and \cite{7}.
\begin{theorem} \label{thm1}
 Let $T$ be a bounded operator on Banach space $X$. If $T$ is a weakly hypercyclic operator, then
\begin{itemize}
\item [(i)] for every norm open set $G$ and weakly open set $W$, there is some integer $n\geq 0$ such that $T^nG \cap W \neq \emptyset$,
\item [(ii)] the set of weakly hypercyclic vectors for $T$ is norm dense in $X$,
\item [(iii)] $T^\ast$ has no eigenvalue,
\item [(iv)] every component of spectrum $T$ intersect the unit circle.
\end{itemize}
\end{theorem}
The following theorem proposes a sufficient condition in terms of open and weakly set for an operator to be weakly hypercyclic. $\ba_N(X)$ and $\ba_N^0 (X)$ represent the subsets $\{x \in X : \| x  \| \leq N \}$ and $\{ x \in X : \| x \| < N \}$ of $X$, respectively, for any integer $N$.
\begin{theorem} \label{thm2}
Let $T$ be a bounded operator on a reflexive Banach space $X$ such that
\begin{itemize}
\item [(i)] for every norm open set $G$ and weakly open set $W$, there is some integer $n\geq 0$ such that $T^n G\cap W \neq \emptyset$,
\item [(ii)] there exists some integer $M$ such that if $G$ is a norm open set with $G \cap \ba(X)\neq \emptyset$ and $W$ is a weakly open set with $W \cap \ba (X)\neq \emptyset$, then there is some integer $n \geq 0$ such that $T^n (G \cap \ba(X))\cap W \cap \ba_M(X)\neq \emptyset$.
\end{itemize}

Then, $T$ is weakly hypercyclic.
\end{theorem}
\begin{proof}
The reflexivity of $X$ and Alaoglu's Theorem imply that, for every integer $N$, $\ba_N(X)$ is weakly compact and metrizable. Let $\Gamma_N$ be a collection of countable bases of $\ba_N(X)$, then $\Gamma = \bigcup \Gamma_N$ is also countable. Suppose $\Gamma = \{W_1,W_2, \cdots, W_N, \cdots \}$. For each $W_j$, there is a weakly open set $V_j$ and an integer $N\geq 1$ such that $W_j = V_j \cap \ba_N(X)$. Let $A=\left\{ (m,j) :T^{-m}V_j \cap \ba^0(X) \neq \emptyset \right\}$. $A$ is nonempty. In fact, condition (i) ensures that there is some integer $m$ such that $(m,j)\in A$ for every $j$. Note that by condition (ii), for every norm open set $G$ and $(m,j) \in A$, there is some integer $n$ such that the set
  $T ^n (G \cap \ba(X)) \cap T^{-m} V_j \cap \ba_M(X)$ is nonempty,  and consequently
  $G \cap \ba(X) \cap T^{-n} \left[  T^{-m} V_j \cap \ba_M(X) \right] \neq \emptyset$. So, the set
  $\bigcup_{n\in \mathbb {N}}T^{-n} \left[ T^{-m} V_j \cap \ba_M(X) \right]$
 is norm dense in $\ba(X)$. The Baire's Category Theorem implies that the set
  $$ WHY:= \bigcap_{(m,j)\in A} \bigcup_n T^{-n} \left[ T^{-m} V_j \cap \ba_M(X) \right] $$
 is also norm dense in $\ba(X)$. Need to be mentioned that to use the Baire's Category Theorem, it would not be difficult to show that
 $T^{-n} \left[ T^{-m} W_j \cap \ba_M(X) \right] \cap \ba(X)$ is open in $\ba(X)$.

  Now, we claim every element of the above set is weakly hypercyclic for $T$. Choosing $x$ in the last set, we show that $x$ is a weakly hypercyclic vector for $T$. For, suppose $W$ is an arbitrary weakly open set in $X$, then by condition (i), $T^m \ba^0(X) \cap W \neq \emptyset$, for some integer $m\geq 0$. There is some integer $N>M \| T \|^m$ such that $T^m \ba^0(X) \cap W \cap \ba_N(X)\neq\emptyset$ and consequently $T^{-m}(W\cap \ba_N(X))\cap \ba^0(X)\neq\emptyset$. Since $\Gamma_N$ is a basis of $\ba_N(X)$ and $\Gamma_N \subseteq \Gamma$, there exists some $W_j \in \Gamma_N$ with $W_j \cap \ba_N(X)\subseteq W\cap \ba_N(X)$. Hence, $T^{-m}(V_j\cap \ba_N(X))\cap \ba^0(X)\neq\emptyset$ when
  $T^{-m} V_j\cap \ba^0(X) \neq \emptyset$ for some $j$. It states that $(m,j)\in A$ and since $x \in WHY(T)$, there is some integer $n$ such that $T^{n+m}x \in V_j$ and $T^nx \in \ba_M(X)$ when $\| T^{n+m}x \| \leq M \| T \|^m \leq N$. So, $T^{n+m}x \in V_j \cap \ba_N(X)$ which is a subset of $W \cap \ba_N(X)$. Thus, $T^{n+m}x \in W$, and so  $\orb(T,x)$ is weakly dense in $X$, and $T$ is weakly hypercyclic. (the subject of orbits has been studied in many literature, for example see \cite{12}.)
 \end{proof}
 %
 %
\section{Weak Hypercyclicity}
\begin{theorem}[Weak Hypercyclicity Criterion]
Let $X$ be a reflexive, separable Banach space, $T \in B(X)$, and there exist two dense subsets $Y$ and $Z$ in $X$, a sequence $\{n_k \}$ of integers and positive integer $M$ such that
\begin{enumerate}
\item $T^{n_k}y\stackrel{wk}{\rightarrow}0$ for every $y \in Y$,
\item for every $y \in Y$ with $\left\|y\right\|\leq 1$, there is some integer $N$ such that $$\sup \{\left\|T^{n_k}y \right\|:k \geq N\}\leq M, $$
\item there exists a linear map $S_k :Z\rightarrow X$ such that for every $z\in Z$,
$$ S_k z\rightarrow 0, \quad \textnormal{and} \quad  T^{n_k}S_{n_k}z \rightarrow z.$$
\end{enumerate}

Then, $T$ is weakly hypercyclic.
\end{theorem}

\begin{proof}
We apply the preceding theorem (Theorem \ref{thm2}). Suppose $G$ is a norm open set and $W$ is a weakly open set. Choose $z \in G\cap Z$, $y \in W\cap Y$, and let $y_k = y+S_k z$, then $y_k\rightarrow y$ in norm topology while $T^{n_k} y_k\stackrel{wk}{\rightarrow}z$. It gives condition (i) of Theorem 1.2.\\
Again suppose $G$ is a norm open set and $W$ is a weakly open set with $W \cap \ba(X)\neq\emptyset$ and $G\cap \ba(X)\neq\emptyset$. Note that $W\cap \ba^0 (X)$ and $G\cap \ba^0 (X)$ are also nonempty. Let $z\in G\cap Y\cap \ba^0(X)$ and $y\in W\cap Y\cap \ba^0(X)$, and $y_k$ be the above sequence in the last paragraph. Then for sufficiently large $k$, $\left\|y_k\right\|\leq 1$, $\left\|T^{n_k}y_k\right\|\leq M+1$ and $y_k\rightarrow y$, $T^{n_k}y_k \stackrel{wk}{\rightarrow} z$ as $k\rightarrow \infty$. It states $T^n G\cap W\cap \ba_{M+1}(X)\neq\emptyset$. Thus, the necessary conditions of part (ii) of Theorem 1.2 are fulfilled and the proof is completed.
\end{proof}

\section{The Unilateral and Bilateral Weighted Shift}
Let $\{e_j : j\in \mathbb{N} \}$ be the canonical basis of $\ell^p(\mathbb{N})$. Then, the operator $T: \ell^p(\mathbb{N}) \rightarrow \ell^p(\mathbb{N})$ defined by $T(e_j)=w_j e_{j-1}$ for $j\geq 2$ and $T(e_1)=0$, for some positive and bounded sequence $\{w_j : j\in \mathbb{N}\}$, is called a \textit{unilateral backward shift}. Also, if $\{ e_j : j\in  \mathbb{Z} \}$ is the standard basis of $\ell^p(\mathbb{Z})$, then we define the bilateral backward shift $T$ on $\ell^p(\mathbb{Z})$ by $T(e_j)=w_j e_{j-1}$ for all $j \in \mathbb{Z}$ and for some positive and bounded weights $\{w_j : j \in \mathbb{Z} \}$.  \\

Salas \cite{13} introduced the norm hypercyclic unilateral and bilateral weighted shift in terms of the sequence of their weights.

\begin{theorem}[Salas Theorem \cite{13}]
\begin{enumerate}
\item The unilateral weighted shift $T$ with weight sequence $\{w_j : j\geq 1\}$ is hypercyclic if and only if $\sup \{w_1 w_2 \cdots w_n : n \geq 1 \} = \infty$.
\item The bilateral weighted shift $T$ with weight sequence $\{w_j : j \in \mathbb{Z}\}$ is hypercyclic if and only if for any given $\epsilon >0$ and $q\in \mathbb{N}$, there exists an arbitrary large $n$ such that for all $| j | < q$,
$$ \prod_{s=1}^n w_{j+s} > \frac{1}{\epsilon}, \quad {\textnormal{and}} \quad \prod_{s=0}^{n-1} w_{j-s} < \epsilon.$$
\end{enumerate}
\end{theorem}

The authors in \cite{5} (Theorem 4.1) showed that $\sup \{w_1 w_2 \cdots w_n : n\geq1 \} = \infty$ is equivalent to weak hypercyclicity, hence, weakly hypercyclic unilateral weighted shift is norm hypercyclic. They also introduced a sufficient condition for a bilateral weighted shift on $\ell^p(\mathbb{Z})$ with $2 \leq p <\infty$ to be weakly hypercyclic. This condition is weaker than Salas' condition and more difficult to state. Someone might claim that the following conjecture provides a sufficient condition for weak hypercyclicity of bilateral weighted shift, based on weak hypercyclicity criterion.

\begin{con}
The bilateral weighted shift $T(e_j)=w_je_{j-1}$ on $\ell^p(\mathbb{Z})$ with $1\leq p <\infty$ is weakly hypercyclic if and only if there exists some sequence $\{n_k\}$ of integers such that
\begin{enumerate}
\item $\sup \{w_jw_{j-1} \cdots w_{j-n_k+1}:k\geq 1,j\in \mathbb{Z}\}<\infty$,
\item for all $j\in \mathbb{Z}$, $w_jw_{j+1} \cdots w_{j+n_k} \rightarrow \infty$ as $k \rightarrow \infty$.
\end{enumerate}
\end{con}

This conjecture is false, because there does not exist a shift that satisfies both conditions (1) and (2) of the conjecture. For a proof, suppose there is an increasing sequence $(n_k)$ satisfying both conditions. Let  $\alpha := \sup \{w_j w_{j-1} \cdots w_{j-n_k+1}:k\geq 1,j\in \mathbb{Z}\} < \infty$. If $\alpha = 0$, then $w_j = 0$ for some $j$, and so condition (2) fails to hold. Assume $\alpha > 0$, then by condition (2), there is $n_k$ such that
\begin{equation} \label{eq1}
w_0 w_1 \cdots w_{n_k} > \alpha N.
\end{equation}

For $j=n_k$, we have
\begin{equation} \label{eq2}
w_{n_k} w_{n_k -1} \cdots w_1 = w_j w_{j-1} \cdots w_{j-n_k +1} \leq \alpha.
\end{equation}

Combining (\ref{eq1}) and (\ref{eq2}) yields
\[
\alpha N < w_0 w_1 \cdots w_{n_k} = w_0 ( w_1 \cdots w_{n_k} ) \leq w_0 \alpha,
\]
and therefore,  $w_0 > N$. Since this inequality holds for any integer $N \geq 1$, we get $w_0 = \infty$ which is a contradiction.

\section{The Linear Fractional Composition Operator}
Every holomorphic self-map $\varphi$ of $U$ induces a linear composition operator $C:H(U)\rightarrow H(U)$ by $C(f)(z)=f(\varphi(z))$ for every $f\in H(U)$ and $z\in U$. Shapiro introduced  a complete characterization of hypercyclic operators on $H(U)$. In fact,  he showed that $C_\varphi$ is hypercyclic if and only if $\varphi$ has no fixed point in $U$. This condition as we see in the theorem below is equivalent to weak hypercyclicity as well.

\begin{theorem}
Let $\varphi$ be a holomorphic self-map on $U$. Then $C_\varphi$ is hypercyclic on $H(U)$ if and only if $C_\varphi$ is weakly hypercyclic.
\end{theorem}

\begin{proof}
Clearly hypercyclicity implies weak hypercyclicity. To prove the other direction, it is enough to show that if $C_\varphi$ is weakly hypercyclic, then $\varphi$ has no fixed point in $U$. Suppose $\varphi$ has a fixed point $p\in U$ and $\orb(C_\varphi ,f)$ is weakly dense for some $f\in H(U)$. Let $g$ be an arbitrary vector in $H(U)$ and $\epsilon >0$.   The linear functional $\Lambda_p:H(U)\rightarrow \mathbb{C}$ by $\Lambda_p(f)=f(p)$ is continuous on $H(U)$, so there is some positive integer $n$ such that $\left|\Lambda_p(f \circ \varphi_n - g)\right|<\epsilon$ when $\left|f(p)-g(p)\right|< \epsilon$. It states that any function in $H(U)$ must have the value $f(p)$ at $p$. But the weakly closure of this orbit cannot be all of $H(U)$ and consequently no $C_\varphi$-orbit is weakly dense, i.e., $C_{\varphi}$ is not weakly hypercyclic.
\end{proof}

Now, suppose $\varphi$ has a linear fractional self-map of $U$ with  no fixed point in $U$. Then, we say $\varphi$ is {\em{parabolic}} if $\varphi$ has only one fixed point which must lie on the unit circle. Parabolic maps are conjugate to translations of the right half-plane into itself. Also, we say $\varphi$ is {\em{hyperbolic}}  if it has two fixed points, one of them lies on the unit circle and the other one is out of $\overline{U}$ which in the automorphism case, both the fixed points must lie on $\partial U$ \cite{4}. \\
Bourdon and Shapiro \cite{3,4} characterized the hypercyclic composition operator on Hardy space, $H^2$. The obtained results state that for a linear fractional self-map $\varphi$ of $U$, $C_\varphi$ is hypercyclic on $H^2$ unless $\varphi$ is a parabolic non-automorphism \cite{4}.

\begin{theorem}
Let $\varphi$ be a linear fractional self-map of $U$. Then $C_\varphi$ is hypercyclic on $H^2$ if and only if $C_\varphi$ is weakly hypercyclic.
\end{theorem}

\begin{proof}
It is clear that if $\varphi$ is not parabolic non-automorphism, then hypercyclicity and weak hypercyclicity for $C_\varphi$ are equivalent. Let $\varphi$ be parabolic non-automorphism, so it has only one fixed point which lies on $\partial U$. Without loss of generality, we may take this fixed point to be $+1$.  Set
$$\sigma(z)=\frac {1+z}{1-z}, \quad  \textnormal{and} \quad   \Phi =\sigma \circ \varphi \circ \sigma^{-1}.$$

Then, $\sigma$ is a linear fractional mapping of $U$ onto the open right half-plane $\mathbf{P}$, and one can easily check that $\Phi (w)=w+a$ where $\RE(a)>0$ and $w\in \mathbf P$. An easy computation shows that for each $z\in U$,
 $$ 1-| \varphi_n(z)|^2=\frac{4 \RE(\sigma (z)+na)}{| 1+\sigma(z)+na| ^2},$$ and
 $$ \varphi_n(z)-\varphi_n(0)=\frac {2(\sigma(z)-\sigma (0))}{(\sigma(z)+na+1)(\sigma(0)+na+1)}.$$

In addition, for each pair of points $z,w \in U$, and $f \in H^2$, the following estimate holds,
 $$ | f(z)-f(w)| \leq 2\left\| f \right\|\frac {| z-w |}{(\min \{1-| w| \ , \ 1-| z | \})^{3/2}}.$$
 By substituting $z$ and $w$ by $\varphi_n(z)$ and $\varphi_n(0)$, respectively,  and using the last estimate, we get
 $$\left|f(\varphi_n(z))-f(\varphi_n(0))\right|\leq \frac {M}{\sqrt n},$$
where the constant $M$ depends on $f$, $z$, and $\varphi$ (for more  details see \cite{3,4}).

Now, suppose $\orb(C_\varphi ,f)$ is weakly dense for some vector $f\in H^2$. The linear functional $\Lambda_p : H^2 \rightarrow \mathbb{C}$ by $\Lambda_p(g)=g(p)$ is bounded for every $p\in U$. Let $g$ be an arbitrary vector in $H^2$, $z\in U$, and $\epsilon >0$. There exists a large enough $n$ such that
  $$\left|f(\varphi_n(z))-f(\varphi_n(0))\right|< \frac{\epsilon}{4},$$
and
$$\left|\Lambda_z(f \circ \varphi_n-g)\right|< \frac{\epsilon}{4}, \quad  {\textnormal{and}} \quad  \left| \Lambda_0 (f \circ \varphi_n-g) \right| < \frac{\epsilon}{4}.$$

Therefore,
$$ \left| f(\varphi_n(z))-g(z) \right| < \frac{\epsilon}{4}, \quad  {\textnormal{and}}  \quad   \left| f(\varphi_n(0)) - g(0) \right| < \frac{\epsilon}{4}. $$

Thus, by the triangle inequality, we can deduce $\left| g(z)-g(0) \right| < \epsilon$, and consequently $g \equiv g(0)$.

Hence, only constant functions can be weak cluster points of the $C_\varphi$-orbit of an $H^2$ function, and thereof  $C_\varphi$ is not weakly hypercyclic. This fulfills the proof.
\end{proof}

\end{document}